\documentclass[12pt]{article}
\usepackage[latin9]{inputenc}
\usepackage{units}
\usepackage{amsmath}
\usepackage{amssymb}
\usepackage{amsthm}

\makeatletter

\usepackage{graphicx} 

\usepackage{amsfonts}
\usepackage{algorithmic}\usepackage{algorithm}
\newtheorem{theo}{Theorem}[section]

\newtheorem{lem}[theo]{Lemma}
\newtheorem{remark}[theo]{Remark}

\newcommand{\bb}[1]{\boldsymbol{\mathrm{#1}}}

\newcommand{\SSS}{\mathbb{S}}

\newcommand{\RR}{\mathbb{R}}
\newcommand{\MM}{\mathbb{M}}

\newcommand{\Aa}{\bb{A}}
\newcommand{\Bb}{\bb{B}}
\newcommand{\Cc}{\bb{C}}

\makeatother

\begin{document}

\title{Asymptotic metrics on the space of matrices under the commutation relation}

\author{Klaus Glashoff and Michael M. Bronstein\\
\small Institute of Computational Science\\
\small Faculty of Informatics,\\
\small Universit{\`a} della Svizzera Italiana\\ 
\small Lugano, Switzerland\\
}
\maketitle
\begin{abstract}
We show that the norm of the commutator defines ``almost a metric'' on the quotient space of commuting matrices, in the sense that it is a semi-metric satisfying the triangle inequality asymptotically for large matrices drawn from a ``good'' distribution. 
We provide theoretical analysis of this results for several distributions of matrices, and show numerical experiments confirming this observation.    
\end{abstract}

\section{Introduction}

Almost commuting matrices have attracted interest since the 1950s, mainly in the field of quantum mechanics, where it was important to establish whether two almost commuting matrices are close to matrices that exactly commute \cite{Bernstein_Almost_commuting,Pearcy1979332,Rordam1996,Huang_Lin,Hastings2009,Filonov2010arXiv1008.4002F,Loring_Sorensen2010,Glebsky2010arXiv1002.3082G}.  
It is well-known that commuting matrices are jointly diagonalizable; 
in \cite{klaus13}, we extended this result to the approximate case, showing that almost commuting matrices are almost jointly diagonalizable. 
This result relates to recent works on methods for approximate joint diagonalization of matrices and their applications  \cite{CardosoBlind1993,Cardoso_pertubation1994,Cardoso96jacobiangles,2012arXiv1209.2295E,KovBBKK:2013:EG}. 
In particular, \cite{KovBBKK:2013:EG} used the joint diagonalizability of matrices as a criterion of their similarity in the context of 3D shape analysis. In light of \cite{klaus13}, we can consider the norm of the commutator instead of performing a computationally expensive approximate joint diagonalization.

In this paper, we are interested in defining a metric between the equivalence classes of commuting matrices using the norm of their commutator. We show that while not a metric, such a construction is a metric asymptotically for sufficiently large matrices with a ``good'' distribution.

\section{Background} 

Let $\Aa, \Bb \in \MM_n(\RR)$ denote two $n\times n$ real matrices, assuming hereinafter $n\geq 2$. We define their {\em commutator} as $[\Aa,\Bb] = \Aa\Bb - \Bb\Aa$. 
In this paper, we study the properties of the Frobenius norm of the commutator, $\| [\Aa,\Bb] \|_\mathrm{F} = \left( \sum_{ij} [\Aa,\Bb]^2_{ij}\right)^{1/2}$. 

A non-negative function $d: \MM_n\times \MM_n \rightarrow \RR_+\cup\{0\}$ is called a {\em metric} if it satisfies the following properties for any $\Aa,\Bb,\Cc \in \MM_n$:

\noindent (M1) {\em Symmetry:} $d(\Aa,\Bb) = d(\Bb,\Aa)$. 

\noindent (M2) {\em Identity:} $d(\Aa,\Bb) = 0$ iff $\Aa = \Bb$. 

\noindent (M3) {\em Triangle inequality:} $d(\Aa,\Bb) + d(\Bb,\Cc) \geq d(\Aa,\Cc)$. 

\noindent In the case when (M2) holds only in one direction ($\Aa = \Bb \Rightarrow d(\Aa,\Bb)=0$), $d$ is called a {\em pseudo-metric}; $d$ satisfying (M1)--(M2) only is called a {\em semi-metric}.

We are interested in a pseudo-metric $d$ on $\MM_n$ satisfying $d(\Aa,\Bb)=0$ for all $\Aa, \Bb \in \MM_n$ such that $[\Aa,\Bb]=0$. Such a pseudo-metric 
can be regarded as measure of the similarity of matrices under which commuting matrices are equivalent. - In the following, we will omit the prefix `pseudo'.

One can easily show that $d(\Aa,\Bb) = \| [\Aa,\Bb]\|_\mathrm{F}$ is not a metric but a semi-metric only, i.e. it violates the triangle inequality (M3): 
a counterexample for $n=2$ is 
\begin{equation*}
\Aa =\left(\begin{array}{cc}
2 & 1\\
1 & 1
\end{array}\right),\,\,\,
\Bb = \left( 
             \begin{array}{cc}
1& 0\\
0 & 1
             \end{array}
\right), \,\,\,
\Cc = \left( 
             \begin{array}{cc}
0 & 1\\
1 & 0
             \end{array}
\right), 
\end{equation*}
for which $\Delta(\Aa,\Bb,\Cc) =  \| [\Aa,\Bb]\|_\mathrm{F} +  \| [\Bb,\Cc]\|_\mathrm{F} -  \| [\Aa,\Cc]\|_\mathrm{F}= -\sqrt{2} < 0$. 

However, when taking matrices $\Aa,\Bb,\Cc$ with i.i.d. normal elements, one obtains the probability distribution of $\Delta(\Aa,\Bb,\Cc)$ as shown in Figure~1; other distributions produce a similar behavior. 
A surprising observation is that for increasing matrix size $n$, the probability of $\Delta(\Aa,\Bb,\Cc)<0$ (i.e., having the triangle inequality violated) decreases. Thus, even though not a metric in the strict sense, the commutator norm $\| [\Aa,\Bb]\|_\mathrm{F}$ behaves like a metric asymptotically.
In the next section, we provide analysis of this behavior.

\begin{figure}
\center{
  \includegraphics[width=1\linewidth]{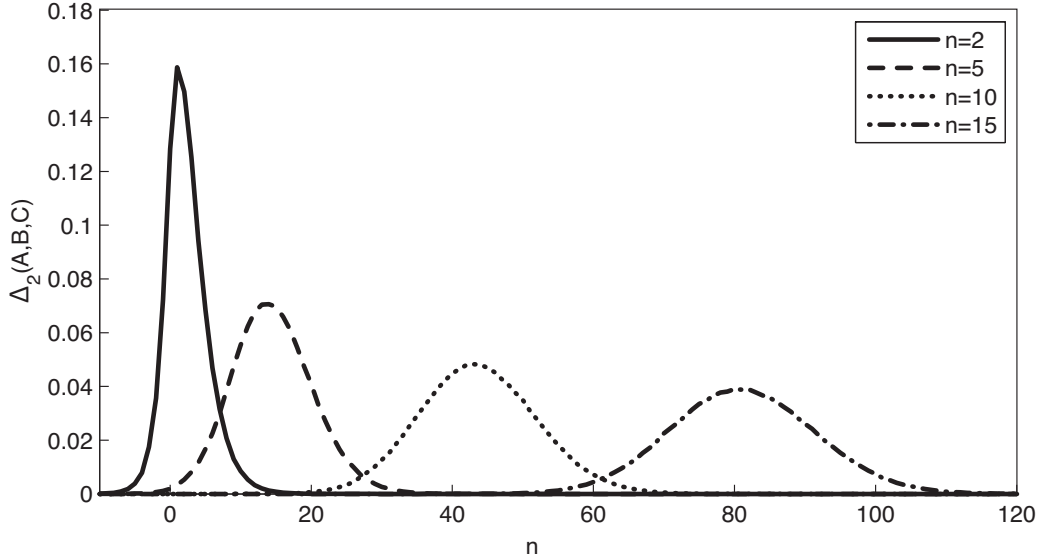}
  \caption{\label{fig:motivation} \small Distribution of  $\Delta(\Aa,\Bb,\Cc)$ for different values of $n$; matrices $\Aa,\Bb,\Cc \in \MM_n$ with normal i.i.d. elements. The negative tail of the distribution decreases with the increase of $n$. }
}
\end{figure}

\section{Asymptotic triangle inequality} 

Let use denote by $\SSS_n = \{ \Aa\in\MM_n(\RR) :  \|\Aa\|_\mathrm{F}=1\}$ the unit sphere of $n\times n$ matrices. 
We consider general distances of the form $d_{\alpha}(\Aa,\Bb) = \|\Aa\Bb - \Bb\Aa\|_\mathrm{F}^\alpha$, where $\alpha >0$, and define the {\em triangle inequality defect} as 
$$\Delta_{\alpha}(\Aa,\Bb,\Cc) = d_{\alpha}(\Aa,\Bb) + d_{\alpha}(\Bb,\Cc) - d_{\alpha}(\Aa,\Cc); $$ 
for $\alpha=1$ we obtain the case discussed in the previous section (using the notation $\Delta = \Delta_{1}$); for $\alpha=2$ we can relate to the results of B{\"ottcher} and Wenzel \cite{Boettcher2005216,Bottcher20081864} who studied the statistical properties of squared norms of matrix commutators. 
%
%

We can formulate our observation in Section 2 as the following 

\begin{theo}[\bf asymptotic triangle inequality] Let $\Aa, \Bb, \Cc$ be independently drawn from a uniform distribution on the unit sphere $\SSS_n$, and $\alpha=1$ or $\alpha=2$. Then, $$\lim_{n\rightarrow \infty} \mathrm{P}(\Delta_{\alpha}(\Aa,\Bb,\Cc)<0)=0.$$  
\end{theo}

\begin{proof} 

For $\alpha = 2$, we use the result of B{\"ottcher} and Wenzel \cite{Bottcher20081864} who showed that the expectation and variance of the squared norm of the commutator under the conditions of the theorem are given by
$$\mathbb{E} \|\Aa\Bb - \Bb\Aa\|_\mathrm{F}^2 = \frac{2}{n} - \frac{2}{n^3}; \,\,\,\,\, \mathrm{Var}( \|\Aa\Bb - \Bb\Aa\|_\mathrm{F}^2) = \frac{8}{n^4} + \mathcal{O}( n^{-5}).$$
%
Since the expectation of a sum of random variables is equal to the sum of the expectations, we get 
\begin{eqnarray*}
\mathbb{E} \Delta_{2}(\Aa,\Bb,\Cc) &=&  \mathbb{E}  \|\Aa\Bb - \Bb\Aa\|_\mathrm{F}^2
= \frac{2}{n} - \frac{2}{n^3}. 
\end{eqnarray*}
Denoting $X_1 = \|\Aa\Bb - \Bb\Aa\|_\mathrm{F}^2, X_2 = \|\Bb\Cc - \Cc\Bb\|_\mathrm{F}^2, X_3 = \|\Aa\Cc - \Cc\Aa\|_\mathrm{F}^2$, the variance is expressed as 
%
$\mathrm{Var} (\Delta_{2}(\Aa,\Bb,\Cc)) = \sum_{ij} \mathrm{Cov}(X_i,X_j)$,  
%
where $\mathrm{Cov}(X_i,X_i) = \mathrm{Var}(X_i)$. 
Using the Cauchy-Schwartz inequality $\mathrm{Cov}(X_i,X_j) \leq \sqrt{\mathrm{Var}(X_i) \mathrm{Var}(X_j)}$, we can bound 
\begin{eqnarray}
\mathrm{Var} (\Delta_{2}(\Aa,\Bb,\Cc)) &=& \frac{72}{n^4} + \mathcal{O}( n^{-5}). 
\end{eqnarray}
Finally, using the Chebychev inequality, we get
\begin{eqnarray*}
\mathrm{P} (\Delta_{2}<0) &\leq& \mathrm{P} (|\Delta_{2} - \mathbb{E}\Delta_{2}|\geq \mathbb{E}\Delta_{2}) 
\leq \frac{\mathrm{Var}(\Delta_{2})}{(\mathbf{E}\Delta_{2})^2} = \mathcal{O}(n^{-2}), 
\end{eqnarray*}
from which the assertion of the theorem follows.

For $\alpha = 1$, we use the following result (the proof is given in the Appendix):

\begin{lem}Let  $\left\{X_n\right\}_{n\ge 1}$ be a sequence of nonnegative random variables with probability distributions $\left\{f_n(x)\right\}_{n\ge 1}$, with expectation $\mu_n = \mathcal{O}(n^{-1})$ and variance $\sigma^2_n = \mathcal{O}(n^{-4})$. 
Then, 
\begin{eqnarray*}
\mathbb{E}(X_n^{1/2}) &=& \mu_n^{1/2}(1+\mathcal{O}(n^{-2})); \\
\mathrm{Var}(X_n^{1/2}) &=&  \mathcal{O}(n^{-3}).
\end{eqnarray*}
\end{lem}

\noindent Applying Lemma 3.2 to $X_n = \|\Aa\Bb - \Bb\Aa \|_\mathrm{F}^2$, we infer that 
$$\mathbb{E}\|\Aa\Bb - \Bb\Aa \|_\mathrm{F}  = \sqrt{2}(n^{-1} - n^{-3})^{1/2}(1+\mathcal{O}(n^{-2})) = \mathcal{O}(n^{-1/2}),$$ 
implying 
$\mathbb{E} \Delta_{1} =  \mathcal{O}(n^{-1/2})$, and 
$\mathrm{Var}(\|\Aa\Bb - \Bb\Aa \|_\mathrm{F}) = \mathcal{O}(n^{-3})$, 
implying $\mathrm{Var}(\Delta_{1}) =  \mathcal{O}(n^{-3})$. 
Therefore, $\mathrm{P} (\Delta_{1}<0) = \mathcal{O}(n^{-2})$, which completes the proof. 

\end{proof}

Table~\ref{tab:triangel_ineq} and Figure~\ref{fig:triang_ineq} show a numerical simulation, experimentally  confirming the asymptotic behavior of the triangle inequality defect $\Delta_{\alpha}$ for different values of $\alpha$.

\begin{figure}
\center{
  \includegraphics[width=1\linewidth]{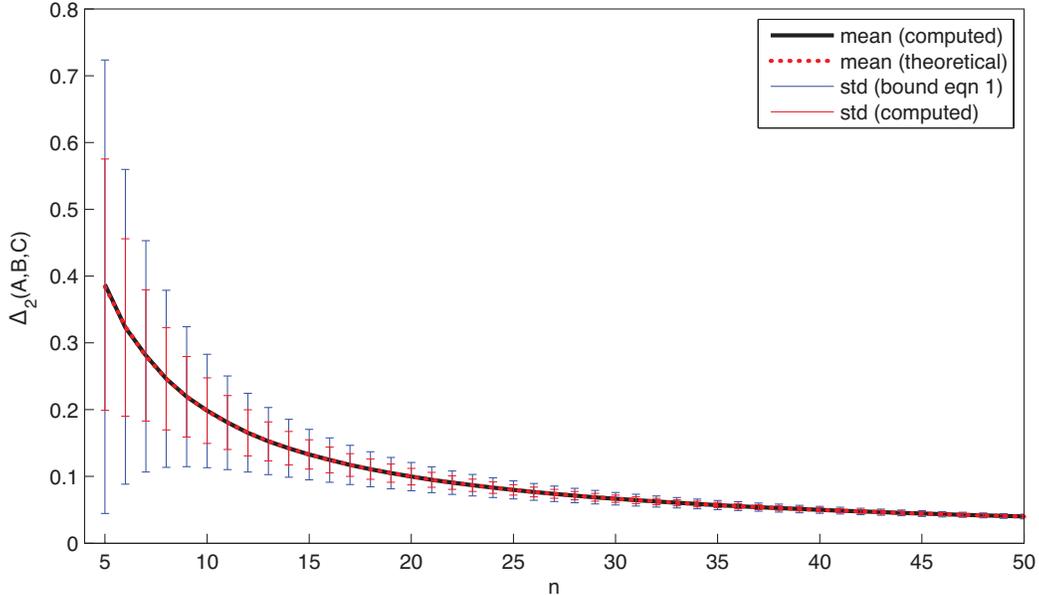}\vspace{-5mm}
  \caption{\label{fig:triang_ineq} \small Expectation and standard deviation (shown are theoretical results and numerical computation on $10^4$ matrices drawn uniformly on the unit sphere) of the triangle inequality defect $\Delta_2$ for different values of $n$.   }
}
\end{figure}

\begin{table*}[htdp]\small
\caption{\small Mean and standard deviation of the triangle inequality defect $\Delta_{\alpha}$ for different values of $n$ and $\alpha$, computed on $10^4$ matrices drawn uniformly on the unit sphere.\vspace{-3mm}
}
\begin{center}
\begin{tabular}{l c c c}
$n$ & $\alpha=0.5$& $\alpha=1$& $\alpha=2$\\
\hline
2 & 0.87 $\pm$ 0.33 					& 0.80 $\pm$ 0.53  					& 0.75 $\pm$ 0.83 \\ 
3 & 0.86 $\pm$ 0.19 					& 0.75 $\pm$ 0.31  					& 0.59 $\pm$ 0.46 \\ 
4 & 0.82 $\pm$ 0.13 					& 0.67 $\pm$ 0.21  					& 0.47 $\pm$ 0.29 \\ 
5 & 0.78 $\pm$ 0.10 					& 0.61 $\pm$ 0.15  					& 0.38 $\pm$ 0.19 \\ 
10 & 0.67 $\pm$ 0.04 					& 0.44 $\pm$ 0.06 					& 0.20 $\pm$ 0.05 \\ 
25 & 0.53 $\pm$ 0.01					& 0.28 $\pm$ 0.01  					& 0.08 $\pm$ 7$\times$10$^{-3}$ \\ 
50 & 0.45 $\pm$ 5$\times$10$^{-3}$ 	& 0.20 $\pm$ 4$\times$10$^{-3}$  	& 0.04 $\pm$ 1$\times$10$^{-3}$ \\ 
100 & 0.38 $\pm$ 2$\times$10$^{-3}$ 	& 0.14 $\pm$ 1$\times$10$^{-3}$  	& 0.02 $\pm$ 4$\times$10$^{-4}$ \\ 
200 & 0.32 $\pm$ 9$\times$10$^{-4}$ 	& 0.10 $\pm$ 6$\times$10$^{-4}$  	& 0.01 $\pm$ 1$\times$10$^{-4}$ \\ 
500 & 0.25 $\pm$ 3$\times$10$^{-4}$ 	& 0.06 $\pm$ 1$\times$10$^{-4}$  	& 4$\times$10$^{-3}$ $\pm$ 2$\times$10$^{-5}$ \\ 
\hline
\end{tabular} 
\end{center}
\label{tab:triangel_ineq}
\end{table*}%

\begin{remark}
Theorem~3.1 holds for other distributions as well. 
B{\"ottcher} and Wenzel \cite{Bottcher20081864} give expressions for the expectation and variance of the squared commutator norm $\|\Aa\Bb - \Bb\Aa\|_\mathrm{F}^2$ for different distributions: 
\begin{enumerate} 
\item $a_{ij}, b_{ij}$ are i.i.d. with standard normal distribution: in this case, 
$$\mathbb{E} \|\Aa\Bb - \Bb\Aa\|_\mathrm{F}^2 = 2n^3-2n; \,\,\,\,\, \mathrm{Var}( \|\Aa\Bb - \Bb\Aa\|_\mathrm{F}^2) =24n^4+ \mathcal{O}( n^{3}),$$
and thus 
$\mathrm{P} (\Delta_{2}<0)  = \mathcal{O}(n^{-2})$. 
\item $a_{ij}, b_{ij}$ are i.i.d. with Rademacher distribution (equi-probable values $\pm1$): the orders of the expectation and variance are the same as in the former case. 
\end{enumerate} 
\end{remark}


%
%
%
%
%
%

%
%

%

\section*{Acknowledgement}

We thank David Wenzel for pointing out an error in the previous version. This research was supported by the ERC Starting Grant No. 307047 (COMET). 

\section*{Appendix} 

\begin{proof}[Proof of Lemma 3.2]
The expectation of $X_n^{1/2}$ is given by 
$$\mathbf{E}X_n^{1/2} = \int_0^\infty \sqrt{x}f_n(x) dx. $$ 
Our evaluation of this integral is based on the `delta method', where we have to take care of the singularity of the second derivative of $\sqrt{x}$ at $x=0$.  
We will need the following inequality several times throughout the proof:
\begin{eqnarray}
\int_{0}^{\frac{\mu_{n}}{2}}f_n(x)dx&=&P(X_n\le {\textstyle \frac{\mu_{n}}{2}}) 
=P(\mu_{n}-X_n\ge {\textstyle \frac{\mu_{n}}{2}})\nonumber\\
&\le&P(\vert \mu_{n}-X_n\vert\ge {\textstyle \frac{\mu_{n}}{2}})
\le 4\frac{\sigma_n^2}{\mu_n^2}=\mathcal{O}(n^{-2}).\label{ineq:first}
\end{eqnarray}
Denote: 
\begin{equation*}
\int_{0}^{\infty}\sqrt{x}f_n(x)dx=\int_{0}^{\frac{\mu_{n}}{2}}\sqrt{x}f_n(x)dx+\int_{\frac{\mu_{n}}{2}}^{\infty}\sqrt{x}f_n(x)dx=I_1+I_2.
\end{equation*}
From our assumption on $\mu_n$ and ($\ref{ineq:first}$), we get
 $$0\le I_1\le\sqrt{\frac{\mu_{n}}{2} }\int_{0}^{\frac{\mu_{n}}{2} }f_n(x)dx = \sqrt{ \mu_n     }   \mathcal{O}(n^{-2}).$$ 

Using the Taylor series for $\sqrt{x}$ at $x_0=\sqrt{\mu_n}$ in the integrand of $I_2$  we obtain
\begin{eqnarray}
I_2&=&\int_{\frac{\mu_{n}}{2}}^{\infty}\left(\sqrt{\mu_n}+\frac{1}{2\sqrt{\mu_n}}(x-\sqrt{\mu_n})+R_1(x) \right)f_n(x)dx\nonumber\\
&=&\sqrt{\mu_n}\int_{\frac{\mu_{n}}{2}}^{\infty}f_n(x)dx+\frac{1}{2\sqrt{\mu_n}}\int_{\frac{\mu_{n}}{2}}^{\infty}(x-\mu_n)f_n(x)dx
+\int_{\frac{\mu_{n}}{2}}^{\infty}R_1(x)f_n(x)dx\nonumber\\
&=&I_{21}+I_{22}+I_{23}.\label{eq:three}
\end{eqnarray}
We have
$I_{21}\le \sqrt{\mu_n}\int_{0}^{\infty}f_n(x)dx=\sqrt{\mu_n}$, and by (\ref{ineq:first}), 
\begin{eqnarray*}
I_{21}&=&\sqrt{\mu_n}P(X_n\ge {\textstyle \frac{\mu_n}{2}}) 
=\sqrt{\mu_n}(1-P(X_n\le {\textstyle \frac{\mu_n}{2}})
\ge\sqrt{\mu_n}(1-4{\textstyle \frac{\sigma_n^2}{\mu_n^2}}). 
\end{eqnarray*}
This yields $\sqrt{\mu_n}(1-4\frac{\sigma_n^2}{\mu_n^2})\le I_{21} \le \sqrt{\mu_n}$, 
and thus
\begin{equation*}
I_{21}=\sqrt{\mu_n}(1+\mathcal{O}(n^{-2})).
\end{equation*}
Employing~(\ref{ineq:first}) again, we have for the second integral of (\ref{eq:three}):
\begin{eqnarray*}
I_{22}&=&\frac{1}{2\sqrt{\mu_n}}\int_{\frac{\mu_{n}}{2}}^{\infty}(x-\mu_n)f_n(x)dx\\
&=&\frac{1}{2\sqrt{\mu_n}}\left (\int_{0}^{\infty}(x-\mu_n)f_n(x)dx-\int_{0}^{\frac{\mu_{n}}{2}}(x-\mu_n)f_n(x)dx \right)\\
&=&\frac{1}{2\sqrt{\mu_n}}\left(0-\int_{0}^{\frac{\mu_{n}}{2}}(x-\mu_n)f_n(x)dx \right)\\
&\le&\frac{1}{2\sqrt{\mu_n}}\mu_n\int_{0}^{\frac{\mu_{n}}{2}}f_n(x)dx\\
&=&\sqrt{\mu_{n}}\mathcal{O}(n^{-2}). 
\end{eqnarray*}

For the third integral of (\ref{eq:three}), we use the Taylor remainder formula for $g(x) = \sqrt{x}$:
\begin{equation*}
R_1(x)=\frac{g^{''}(\xi)}{2!}(x-\mu_n)^2 =-\frac{1}{8}\xi^{-3/2}(x-\mu_n)^2, 
\end{equation*}
for some $\xi \in [\frac{\mu_n}{2}, \infty]$ depending on $x$. It then follows that 
\begin{eqnarray*}
|I_{23}| &=& \int_{\frac{\mu_{n}}{2}}^{\infty}|R_1(x)|f_n(x)dx 
= \int_{\frac{\mu_{n}}{2}}^{\infty} \frac{\xi^{-3/2}}{8} (x-\mu_n)^2f_n(x)dx\\
&\le& \max_{\xi \in [\frac{\mu_n}{2},\infty]}\frac{\xi^{-3/2}}{8} \int_{0}^{\infty}(x-\mu_n)^2f_n(x)dx = \frac{1}{\sqrt{8}} \sqrt{\mu_n} \mu_n^{-2} \sigma_n^2 = \sqrt{\mu_n}\mathcal{O}(n^{-2}).
\end{eqnarray*}
%


Combining the results, we get $I_1 + I_2 = \sqrt{\mu_n}(1 + \mathcal{O}(n^{-2}))$, which proves the Lemma concerning the expectation.
For the variance, we use the relation 
\begin{eqnarray*}
\mathrm{Var} (X_n^{1/2}) &=& \mathbb{E} X_n - (\mathbb{E}X_n^{1/2})^2 = \mu_n - \mu_n (1 + \mathcal{O}(n^{-2}))^2 \\
&=& \mu_n(\mathcal{O}(n^{-2}) + \mathcal{O}(n^{-4})) = \mathcal{O}(n^{-3}). 
\end{eqnarray*}

\end{proof}

\bibliographystyle{plain}\small
\bibliography{Spectral.bib}

\end{document}